\documentclass[a4paper,english]{article}
\usepackage[a4paper]{geometry}
\usepackage{amsthm,amsmath,amssymb,mathtools}
\usepackage{enumerate,enumitem,csquotes,verbatim}
\usepackage{setspace}
\usepackage[numbers,sort]{natbib}
\usepackage{hyperref}
\usepackage{color}
\usepackage{tikz}

\theoremstyle{plain}
\newtheorem*{theorem*}{Theorem}
\newtheorem{theorem}{Theorem}
\newtheorem{lemma}[theorem]{Lemma}

\newtheorem*{claim*}{Claim}

\newtheorem{conjecture}[theorem]{Conjecture}

\theoremstyle{definition}

\theoremstyle{remark}

\def\PP{\mathcal{P}}

\def\A{\mathcal{A}}

\let\emptyset\varnothing

\title{Union vertex-distinguishing edge colorings}

\author{Teeradej Kittipassorn\thanks{\,Department of Mathematics and Computer Science, Faculty of Science, Chulalongkorn University, Bangkok 10330, Thailand; \texttt{teeradej.k@chula.ac.th}.}
  \and Preechaya Sanyatit\thanks{\,Department of Mathematics, Faculty of Science, Silpakorn University, Nakhon Pathom 73000 Thailand; \texttt{sanyatit\_p@silpakorn.edu}; \emph{Corresponding author}.}}

\begin{document}
\maketitle

\begin{abstract}
The \emph{union vertex-distinguishing chromatic index} $\chi'_\cup(G)$ of a graph $G$ is the smallest natural number $k$ such that the edges of $G$ can be assigned nonempty subsets of $[k]$ so that the union of the subsets assigned to the edges incident to each vertex is different.
We prove that $\chi'_\cup(G) \in \left\{ \left\lceil \log_2\left(n +1\right) \right\rceil, \left\lceil \log_2\left(n +1\right) \right\rceil+1 \right\}$ for a graph $G$ on $n$ vertices without a component of order at most two.
This answers a question posed by Bousquet, Dailly, Duch\^{e}ne, Kheddouci and Parreau, and independently by Chartrand, Hallas and Zhang.
\end{abstract}

\section{Introduction}

An edge coloring of a graph is \emph{vertex-distinguishing} if the set of colors of the edges incident to each vertex is different.
The pioneers of this concept include Harary and Plantholt~\cite{Harary}, Aigner, Triesch and Tuza~\cite{Tuza}, and Burris and Schelp~\cite{Burris1997}.
It has been widely studied for arbitrary edge colorings~\cite{Aigner,Burris1995}
and also for proper edge colorings~\cite{Bazgan,Balister,Riordan}.
The weaker notion where only adjacent vertices must be distinguished
has also been considered for both arbitrary edge colorings~\cite{Gyori,Hornak}
and for proper edge colorings~\cite{Zhang,Hatami,Lehel}.
Many variants have been investigated where the set of colors is replaced by the multiset of colors~\cite{Burris1994,Escuadro,Radcliffe}, the sum~\cite{Thomason,Seamone,Jacobson}, the difference~\cite{Tahraoui} or the product~\cite{Skowronek}.

In this paper, we are interested in the generalization where the colors are nonempty subsets of $[k]=\{1,2,\dots,k\}$ and we distinguish each vertex by the union of the colors of the edges incident to it.
Given a graph $G$ on $n$ vertices and a natural number $k$, an edge coloring $c:E(G) \rightarrow \PP([k]) \setminus \{\emptyset\}$ is \emph{union vertex-distinguishing} if its \emph{union vertex coloring} $c_\cup:V(G) \rightarrow \PP([k]) \setminus \{\emptyset\}$ given by
\[c_\cup(v)=\bigcup_{uv \in E(G)} c(uv),\]
is injective, i.e. $c_\cup(u) \not= c_\cup(v)$ for all distinct vertices $u$ and $v$.
The \emph{union vertex-distinguishing chromatic index} $\chi'_\cup(G)$ of $G$ is the smallest natural number $k$ such that $G$ admits a union vertex-distinguishing edge coloring.
This concept was introduced by Bousquet, Dailly, Duch\^{e}ne, Kheddouci and Parreau~\cite{Bousquet}, and independently by Chartrand, Hallas and Zhang~\cite{Hallas} under the name \emph{strong royal coloring}.

We shall only consider graphs without a component of order at most two otherwise a union vertex-distinguishing edge coloring does not exist.
Observe that $\chi'_\cup(G) \ge \left\lceil \log_2\left(n +1\right) \right\rceil$ since, in $c_\cup:V(G) \rightarrow \PP([k]) \setminus \{\emptyset\}$, there must be at least as many colors as vertices, i.e. $2^k-1 \ge n$.
On the other hand, Bousquet, Dailly, Duch\^{e}ne, Kheddouci and Parreau~\cite{Bousquet} proved that there are only three possible values for
\[\chi'_\cup(G) \in \left\{ \left\lceil \log_2\left(n +1\right) \right\rceil, \left\lceil \log_2\left(n +1\right) \right\rceil+1, \left\lceil \log_2\left(n +1\right) \right\rceil+2 \right\}.\]
Paths, complete binary trees, cycles (except $C_3,C_7$) and hypercubes satisfy $\chi'_\cup(G) = \left\lceil \log_2\left(n +1\right) \right\rceil$,
while complete graphs (except $K_{2^k}$ for all $k$) and graphs with large minimum degree satisfy $\chi'_\cup(G) = \left\lceil \log_2\left(n +1\right) \right\rceil +1$ (see~\cite{Bousquet,Hallas,Ali}).
However, no graphs $G$ with $\chi'_\cup(G) = \left\lceil \log_2\left(n +1\right) \right\rceil+2$ have been found.
Bousquet, Dailly, Duch\^{e}ne, Kheddouci and Parreau~\cite{Bousquet} asked whether they exist, while
Chartrand, Hallas and Zhang~\cite{Hallas} conjectured that they do not exist.
We prove this conjecture.

\begin{theorem}\label{thm:main}
$\chi'_\cup(G) \in \left\{ \left\lceil \log_2\left(n +1\right) \right\rceil, \left\lceil \log_2\left(n +1\right) \right\rceil+1 \right\}$ for any graph $G$ on $n$ vertices without a component of order at most two.
\end{theorem}

To prove Theorem~\ref{thm:main}, it is enough to prove the following result.
This reduction was made by Bousquet, Dailly, Duch\^{e}ne, Kheddouci and Parreau~\cite{Bousquet}.
A \emph{$1$-star} is a graph obtained from a star on at least three vertices by subdividing each edge at most once.

\begin{theorem}\label{thm:forest}
$\chi'_\cup(F)=\left\lceil \log_2\left(n +1\right) \right\rceil$ for any forest of $1$-stars $F$ on $n$ vertices.
\end{theorem}

The rest of this paper is organized as follows.
In the next section, we shall prove Theorem~\ref{thm:forest} and deduce Theorem~\ref{thm:main} from it.
We then conclude with some remarks and open problems.

\section{Proofs of the main results}
\label{sec:proof}

First, we show how to deduce Theorem~\ref{thm:main} from Theorem~\ref{thm:forest}.

\begin{proof}[Proof of Theorem~\ref{thm:main}]
Let $G$ be a graph on $n$ vertices without a component of order at most two.
We claim that $G$ has a spanning forest of $1$-stars $F$.
Indeed, consider a minimal spanning forest of $G$ without a tree of order at most two, and so the deletion of any edge creates a tree of order at most two.
We shall show that any tree $T$ in the forest is a $1$-star.
If there is a vertex $v$ in $T$ of degree at least three, then $T$ is a $1$-star since, for each neighbor $u$ of $v$, the component of $u$ in $T-uv$ has order at most two.
So we may assume that all vertices in $T$ have degree at most two, i.e. $T$ is a path.
Then $\lvert V(T) \rvert \le 5$ otherwise the deletion of the middle edge does not create a tree of order at most two.
Therefore, $T$ is a $1$-star.

Let $k=\left\lceil \log_2\left(n +1\right) \right\rceil$.
By Theorem~\ref{thm:forest}, there exists an edge coloring $c:E(F) \rightarrow \PP([k]) \setminus \{\emptyset\}$ whose union vertex coloring $c_\cup$ is injective.
We extend $c$ to $c':E(G) \rightarrow \PP([k+1]) \setminus \{\emptyset\}$ by
\[c'(e) = \begin{cases}
c(e) &\text{for }e \in E(F),\\
\{k+1\} &\text{otherwise}.
\end{cases}\]
Then its union vertex coloring $c'_\cup$ is also injective since $c'_\cup(v) \cap [k] = c_\cup(v)$ for all $v \in V(G)$.
Therefore, $\chi'_\cup(G) \le k+1$.
\end{proof}

To prove Theorem~\ref{thm:forest}, the key idea is to partition $\PP([k]) \setminus \{\emptyset\}$ into collections of arbitrary sizes with appropriate properties.
We would like each collection to be the image of the vertices of each $1$-star under the union vertex coloring.
We choose the properties of a collection of size $m$ in such a way that, for any $1$-star of order $m$, there exists an edge coloring whose union vertex coloring maps the vertices of the $1$-star onto the collection.

A collection of distinct sets $A_1,A_2,\dots,A_m$ is an \emph{$m$-star} if
\begin{itemize}
\item for $m\ge4$, $A_1=A_2 \cup A_4$ and $A_{2i+1} \subset A_{2i} \subset A_1$ for each $1 \le i \le \frac{m}{2}$,
\item for $m=3$, $A_1 = A_2 \cup A_3$,
\item for $m=2$, $A_2 \subset A_1$,
\item for $m=1$, no condition.
\end{itemize}
We also allow $m$ to be $1,2$ as we shall prove the existence of such a partition by induction.
A collection of sets is an \emph{$(m_1,m_2,\dots,m_r)$-forest} if it can be partitioned into an $m_1$-star, an $m_2$-star, $\dots$, and an $m_r$-star.

\begin{theorem}\label{thm:set}
If $m_1+m_2+\dots+m_r = 2^k-1$ then
$\PP([k]) \setminus \{\emptyset\}$ is an $(m_1,m_2,\dots,m_r)$-forest.
\end{theorem}

Now, we show how to deduce Theorem~\ref{thm:forest} from Theorem~\ref{thm:set}.

\begin{proof}[Proof of Theorem~\ref{thm:forest}]
Let $F$ be a forest of $1$-stars of orders $m_1,m_2,\dots,m_r \ge 3$ with $n=m_1+m_2+\dots+m_r$,
and let $k=\left\lceil \log_2\left(n +1\right) \right\rceil$.
Since $m_1+m_2+\dots+m_r \le 2^k-1$, by Theorem~\ref{thm:set},
the collection $\PP([k]) \setminus \{\emptyset\}$ can be partitioned into an $m_1$-star, an $m_2$-star, $\dots$, an $m_r$-star and a $(2^k-1-n)$-star.
To prove that $\chi'_\cup(F) \le k$, it is enough to find an edge coloring of $F$ whose union vertex coloring maps the $1$-star of order $m_i$ onto the $m_i$-star for each $i$.

It remains to show that, given an $m$-star $\A=\{A_1,A_2,\dots,A_m\} \subset \PP([k]) \setminus \{\emptyset\}$ with $m \ge 3$, the edges of any $1$-star of order $m$ can be assigned nonempty subsets of $[k]$ so that the unions of the subsets assigned to the edges incident to each vertex, are $A_1,A_2,\dots,A_m$.
Consider a $1$-star of order $m$ obtained from a star with vertices $v,v_1,v_2,\dots,v_{m-l-1}$ where $v$ is the center, by subdividing the edges $vv_1,vv_2,\dots,vv_l$ with vertices $u_1,u_2,\dots,u_l$ respectively.
To obtain $A_1,A_2,\dots,A_m$ as the unions at the vertices
\[v,u_1,v_1,u_2,v_2,\dots,u_l,v_l,v_{l+1},v_{l+2},\dots,v_{m-l-1}\]
respectively,
we assign the subsets $A_2,A_3,\dots,A_m$ to the edges
\[vu_1,u_1v_1,vu_2,u_2v_2,\dots,vu_l,u_lv_l,vv_{l+1},vv_{l+2},\dots,vv_{m-l-1}\]
respectively. Indeed, the union vertex coloring $c_\cup$ satisfies
\begin{align*}
c_\cup(v)
&= A_2 \cup A_4 \cup \dots \cup A_{2l} \cup A_{2l+2} \cup A_{2l+3} \cup \dots \cup A_m = A_1,\\
c_\cup(u_i)
&= A_{2i} \cup A_{2i+1} = A_{2i} \quad \text{for }1 \le i \le l,\\
c_\cup(v_i) &=
\begin{cases}
A_{2i+1} &\text{for }1 \le i \le l,\\
A_{i+l+1} &\text{for }l+1 \le i \le m-l-1.
\end{cases}
\end{align*}
by the properties of the $m$-star.
\end{proof}

To prove Theorem~\ref{thm:set}, we shall construct new stars from existing ones. For instance, we show that if we double an $m$-star by adding a new element to every set then the new collection can be partitioned into an $i$-star and a $(2m-i)$-star for any odd number $i$.
For a collection $\A \subset \PP([k-1]) \setminus \{\emptyset\}$, we write $\A' = \{A\cup\{k\}: A\in\A\} \subset \PP([k]) \setminus \{\emptyset\}$.

\begin{lemma}\label{lem:double}
If $\A \subset \PP([k-1]) \setminus \{\emptyset\}$ is an $m$-star then
\begin{enumerate}[label={$(\roman*)$}, ref={\thelemma$(\roman*)$}]
\item $\A\cup\A'$ is an $(i,2m-i)$-forest for each odd number $1\le i\le m$,
\item $\A\cup\A'\cup \{\{k\}\}$ is a $(2m+1)$-star, and
\item $\A\cup\A'$ is a $2m$-star.
\end{enumerate}
\end{lemma}

\begin{proof}
$(i)$ Suppose first that $m\ge6$ is even.
Let $\A=\{A_1,A_2,\dots,A_m\}$ be such that $A_1=A_2 \cup A_4$ and $A_{2j+1} \subset A_{2j} \subset A_1$ for each $1 \le j \le \frac{m}{2}$.
Then the first $i$ sets form an $i$-star if $i\not=3$, and the remaining $2m-i$ sets in $\A\cup\A'$ can be ordered to form a $(2m-i)$-star as follows:
\[A'_1,A'_2,A'_3,A'_4,A'_5,\dots,A'_{m-2},A'_{m-1},A_{i+1},A_{i+2},\dots,A_{m-2},A_{m-1},A'_m,A_m\]
where $A'_i = A_i\cup\{k\}$.

For $i=3$, the sets $A'_2,A'_3,A_2$ form a $3$-star and 
\[A'_1,A_1,A_3,A'_4,A'_5,A'_6,A'_7,\dots,A'_{m-2},A'_{m-1},A_4,A_5,\dots,A_{m-2},A_{m-1},A'_m,A_m\]
form a $(2m-3)$-star.

Suppose now that $m\ge5$ is odd. Take the above partition for $\lvert \A \rvert = m+1$ and delete the final two sets $A'_{m+1},A_{m+1}$ from the second star.

For the remaining case $m\le4$, we partition $\A\cup\A'$ as follows:
\begin{itemize}
    \item for $m=4$ and $i=1$, $A_1$ and $A'_1,A'_2,A'_3,A'_4,A_4,A_2,A_3$,
    \item for $m=4$ and $i=3$, $A'_2,A'_3,A_2$ and $A'_1,A_1,A_3,A'_4,A_4$,
    \item for $m=3$ and $i=1$, $A_1$ and $A'_1,A'_2,A_2,A'_3,A_3$,
    \item for $m=3$ and $i=3$, $A_1,A_2,A_3$ and $A'_1,A'_2,A'_3$,
    \item for $m=2$ and $i=1$, $A_2$ and $A'_1,A'_2,A_1$,
    \item for $m=1$ and $i=1$, $A_1$ and $A'_1$.
\end{itemize}

$(ii)$ Let $\A=\{A_1,A_2,\dots,A_m\}$ be an $m$-star for $m\ge2$.
Then the sets in $\A\cup\A' \cup \{\{k\}\}$ can be ordered to form a $(2m+1)$-star as follows:
\[A'_1,A_1,A_m,A'_2,A_2,A'_3,A_3,\dots,A'_{m-1},A_{m-1},A'_m,\{k\}.\]
For $m=1$, the sets $A'_1,A_1,\{k\}$ form a $3$-star.

$(iii)$ Take the ordering in $(ii)$ and delete the final set $\{k\}$.
\end{proof}

We are now ready to prove Theorem~\ref{thm:set}.

\begin{proof}[Proof of Theorem~\ref{thm:set}]
We proceed by induction on $k$.
The result is trivial for $k=1$.
Suppose that $m_1+m_2+\dots+m_r = 2^k-1$.
Since there are an odd number of odd $m_i$'s, without loss of generality, we may assume that $m_1,m_2,\dots,m_{2s+1}$ are odd and $m_{2s+2},m_{2s+3},\dots,m_r$ are even.
Rewrite
\begin{align*}
2^{k-1}-1
&= \frac12\left(m_1+m_2+\dots+m_r-1\right)\\
&= \frac{m_1+m_2}{2}+\frac{m_3+m_4}{2}+\dots+\frac{m_{2s+1}-1}{2}+\frac{m_{2s+2}}{2}+\frac{m_{2s+3}}{2}+\dots+\frac{m_r}{2}\\
&= m'_1+m'_2+\dots+m'_{s+1}+m'_{s+2}+m'_{s+3}+\dots+m'_{r-s}
\end{align*}
where the summands in the second line are defined to be the $m'_i$'s in the third line.
By the induction hypothesis, $\PP([k-1]) \setminus \{\emptyset\}$ is an $(m'_1,m'_2,\dots,m'_{r-s})$-forest, i.e. there is a partition $\PP([k-1]) \setminus \{\emptyset\} = \A_1 \cup \A_2 \cup \dots \cup \A_{r-s}$ such that $\A_i$ is an $m'_i$-star for each $i$.
By Lemma~\ref{lem:double}, we have
\begin{enumerate}[label={$(\roman*)$}, ref={\thelemma$(\roman*)$}]
\item $\A_i \cup \A'_i$ is an $(m_{2i-1},m_{2i})$-forest for each $1\le i \le s$,
\item $\A_{s+1} \cup \A'_{s+1} \cup \{\{k\}\}$ is an $m_{2s+1}$-star, and
\item $\A_i \cup \A'_i$ is an $m_{i+s}$-star for each $s+2\le i \le r-s$.
\end{enumerate}
Since the above collections form a partition of $\PP([k]) \setminus \{\emptyset\}$, it is an $(m_1,m_2,\dots,m_r)$-forest.
\end{proof}

\section{Concluding remarks}
\label{sec:conclude}

What if we are allowed to color the edges with the empty set as well?
It turns out our argument can be adapted to show that every graph can be optimally colored, i.e. $\chi'_\cup(G) = \left\lceil \log_2 n \right\rceil$ for any graph $G$ on $n$ vertices without a component of order at most two.
Indeed, in the proof of Theorem~\ref{thm:main}, we color the edges outside the spanning forest of $1$-stars with $\emptyset$ instead of $\{k+1\}$, and Theorem~\ref{thm:set} is modified to state that if $m_1+m_2+\dots+m_r = 2^k$ then
$\PP([k])$ is an $(m_1,m_2,\dots,m_r)$-forest.

We now know that $\chi'_\cup(G) \in \left\{ \left\lceil \log_2\left(n +1\right) \right\rceil, \left\lceil \log_2\left(n +1\right) \right\rceil+1 \right\}$.
It would be interesting to characterize the graphs that can be optimally colored.
This is open even for trees (see~\cite{Bousquet,Hallas}).
We think that forests can be optimally colored.

\begin{conjecture}
$\chi'_\cup(F)=\left\lceil \log_2\left(n +1\right) \right\rceil$ for any forest $F$ on $n$ vertices.
\end{conjecture}

Another direction forward could be to study natural variants including the case where the edge coloring must be proper and the case where only adjacent vertices must be distinguished.

\bibliographystyle{siam}
\bibliography{main}

\begin{thebibliography}{10}

\bibitem{Aigner}
{\sc M.~Aigner and E.~Triesch}, {\em Irregular assignments and two problems \`a
  la {R}ingel}, in Topics in combinatorics and graph theory ({O}berwolfach,
  1990), Physica, Heidelberg, 1990, pp.~29--36.

\bibitem{Tuza}
{\sc M.~Aigner, E.~Triesch, and Z.~Tuza}, {\em Irregular assignments and
  vertex-distinguishing edge-colorings of graphs}, in Combinatorics '90
  ({G}aeta, 1990), vol.~52 of Ann. Discrete Math., North-Holland, Amsterdam,
  1992, pp.~1--9.

\bibitem{Ali}
{\sc A.~Ali, G.~Chartrand, J.~Hallas, and P.~Zhang}, {\em Extremal problems in
  royal colorings of graphs}, J. Combin. Math. Combin. Comput., 116 (2021),
  pp.~201--218.

\bibitem{Balister}
{\sc P.~N. Balister, B.~Bollob\'{a}s, and R.~H. Schelp}, {\em Vertex
  distinguishing colorings of graphs with {$\Delta(G)=2$}}, Discrete Math., 252
  (2002), pp.~17--29.

\bibitem{Lehel}
{\sc P.~N. Balister, E.~Gy\H{o}ri, J.~Lehel, and R.~H. Schelp}, {\em Adjacent
  vertex distinguishing edge-colorings}, SIAM J. Discrete Math., 21 (2007),
  pp.~237--250.

\bibitem{Riordan}
{\sc P.~N. Balister, O.~M. Riordan, and R.~H. Schelp}, {\em
  Vertex-distinguishing edge colorings of graphs}, J. Graph Theory, 42 (2003),
  pp.~95--109.

\bibitem{Bazgan}
{\sc C.~Bazgan, A.~Harkat-Benhamdine, H.~Li, and M.~Wo\'{z}niak}, {\em On the
  vertex-distinguishing proper edge-colorings of graphs}, J. Combin. Theory
  Ser. B, 75 (1999), pp.~288--301.

\bibitem{Bousquet}
{\sc N.~Bousquet, A.~Dailly, E.~Duch\^{e}ne, H.~Kheddouci, and A.~Parreau},
  {\em A {V}izing-like theorem for union vertex-distinguishing edge coloring},
  Discrete Appl. Math., 232 (2017), pp.~88--98.

\bibitem{Burris1994}
{\sc A.~C. Burris}, {\em On graphs with irregular coloring number {$2$}}, in
  Proceedings of the {T}wenty-fifth {S}outheastern {I}nternational {C}onference
  on {C}ombinatorics, {G}raph {T}heory and {C}omputing ({B}oca {R}aton, {FL},
  1994), vol.~100, 1994, pp.~129--140.

\bibitem{Burris1995}
{\sc A.~C. Burris}, {\em The irregular coloring number of a tree}, Discrete
  Math., 141 (1995), pp.~279--283.

\bibitem{Burris1997}
{\sc A.~C. Burris and R.~H. Schelp}, {\em Vertex-distinguishing proper
  edge-colorings}, J. Graph Theory, 26 (1997), pp.~73--82.

\bibitem{Escuadro}
{\sc G.~Chartrand, H.~Escuadro, F.~Okamoto, and P.~Zhang}, {\em Detectable
  colorings of graphs}, Util. Math., 69 (2006), pp.~13--32.

\bibitem{Hallas}
{\sc G.~Chartrand, J.~Hallas, and P.~Zhang}, {\em Royal colorings of graphs},
  Ars Combin.
\newblock to appear.

\bibitem{Jacobson}
{\sc G.~Chartrand, M.~S. Jacobson, J.~Lehel, O.~R. Oellermann, S.~Ruiz, and
  F.~Saba}, {\em Irregular networks}, vol.~64, 1988, pp.~197--210.
\newblock 250th Anniversary Conference on Graph Theory (Fort Wayne, IN, 1986).

\bibitem{Gyori}
{\sc E.~Gy\H{o}ri, M.~Hor\v{n}\'{a}k, C.~Palmer, and M.~Wo\'{z}niak}, {\em
  General neighbour-distinguishing index of a graph}, Discrete Math., 308
  (2008), pp.~827--831.

\bibitem{Harary}
{\sc F.~Harary and M.~Plantholt}, {\em The point-distinguishing chromatic
  index}, in Graphs and applications ({B}oulder, {C}olo., 1982),
  Wiley-Intersci. Publ., Wiley, New York, 1985, pp.~147--162.

\bibitem{Hatami}
{\sc H.~Hatami}, {\em {$\Delta+300$} is a bound on the adjacent vertex
  distinguishing edge chromatic number}, J. Combin. Theory Ser. B, 95 (2005),
  pp.~246--256.

\bibitem{Hornak}
{\sc M.~Hor\v{n}\'{a}k and R.~Sot\'{a}k}, {\em General neighbour-distinguishing
  index via chromatic number}, Discrete Math., 310 (2010), pp.~1733--1736.

\bibitem{Thomason}
{\sc M.~Karo\'{n}ski, T.~\L~uczak, and A.~Thomason}, {\em Edge weights and
  vertex colours}, J. Combin. Theory Ser. B, 91 (2004), pp.~151--157.

\bibitem{Radcliffe}
{\sc M.~Radcliffe and P.~Zhang}, {\em On irregular colorings of graphs}, AKCE
  Int. J. Graphs Comb., 3 (2006), pp.~175--191.

\bibitem{Seamone}
{\sc B.~Seamone}, {\em The 1-2-3 conjecture and related problems: a survey},
  arXiv preprint arXiv:1211.5122,  (2012).

\bibitem{Skowronek}
{\sc J.~Skowronek-Kazi\'{o}w}, {\em Multiplicative vertex-colouring weightings
  of graphs}, Inform. Process. Lett., 112 (2012), pp.~191--194.

\bibitem{Tahraoui}
{\sc M.~A. Tahraoui, E.~Duch\^{e}ne, and H.~Kheddouci}, {\em Gap
  vertex-distinguishing edge colorings of graphs}, Discrete Math., 312 (2012),
  pp.~3011--3025.

\bibitem{Zhang}
{\sc Z.~Zhang, L.~Liu, and J.~Wang}, {\em Adjacent strong edge coloring of
  graphs}, Appl. Math. Lett., 15 (2002), pp.~623--626.

\end{thebibliography}

\end{document}